\documentclass[12pt]{article}
\usepackage{authblk}
\author[1]{Asl{\i} Deniz}
\author[2]{Carsten Lunde Petersen} 

\affil[1]{American University of the Middle East, CBA, Department of Mathematics and Statistics, Kuwait}

\affil[2]{Roskilde University\\

Department of Science and Environment, Denmark
}
\title{Holomorphic Explosions - I}
\date{}
\usepackage[english]{babel}
\usepackage{graphicx}
\usepackage{amscd}
\usepackage{amsfonts}
\usepackage{amscd}
\usepackage{latexsym}
\usepackage{epsfig}
\usepackage{amssymb,amsmath}
\usepackage[usenames]{color}
\usepackage{amscd}
\usepackage{psfrag}
\usepackage{amsthm}
\usepackage{bm}
\usepackage{tikz-cd}
\usepackage[linktocpage=true]{hyperref}
\theoremstyle{plain}

\def\la{\lambda}
\def\La{\Lambda}
\def\sm{\setminus}
\def\CC{{\mathcal{C}}}
\def\HH{{\mathcal{H}}}
\def\OO{{\mathcal{O}}}
\def\UU{{\mathcal{U}}}

\def\AAA{{\mathbb{A}}}
\newcommand{\D}{{\mathbb{D}}}
\newcommand{\Dbar}{{\overline{\mathbb{D}}}}
\newcommand{\Dstar}{{\mathbb{D}^*}}
\newcommand{\Chat}{{\widehat{\mathbb{C}}}}

\newcommand{\whpsi}{{\widehat{\psi}}}

\newcommand{\Z}{{\mathbb{Z}}}

  \newcommand*{\Ala}{{\bm A_\la}}

  \newcommand*{\C}{\mathbb{C}}
  \newcommand*{\Cstar}{\ensuremath{{\mathbb  C}^*}}

  \newcommand*{\ga}{\gamma}

  \newcommand*{\GlaA}{{G_{\la,A}}}

  \newcommand*{\laz}{{\lambda_0}}
  
  \newcommand*{\labar}{{\overline{\lambda}}}
  \newcommand*{\lastar}{{\lambda^*}}
  
  \newcommand*{\Lastar}{{\Lambda^*}}

  \newcommand*{\mapfromto}[3]{\hbox{\ensuremath{#1 : #2 \longrightarrow #3}}}
  
  \newcommand*{\Mbrot}{{\mathrm{M}}}
  
  \newcommand*{\Mla}{{\Mbrot_\la}}

  \newcommand*{\Mobius}{{M{\"o}bius}}

  \newcommand*{\MMtwo}{{\mathcal{M}_2}}

  \newcommand*{\Per}{{\operatorname{Per}}}

  \newcommand*{\si}{\sigma}
  \newcommand*{\sibf}{{\bm\si}}

  \newcommand*{\whH}{\ensuremath{{\widehat{H}}}}

  \newcommand*{\whMbrot}{{\widehat{\Mbrot}}}

  \newcommand*{\wtH}{\ensuremath{{\widetilde{H}}}}

\newcommand{\ALIGN}{\begin{align*}}
\newcommand{\ENDALIGN}{\end{align*}}
\newcommand{\ENUM}{\begin{enumerate}}
\newcommand{\ENUMa}{\begin{enumerate}[a.]}
\newcommand{\ENUMA}{\begin{enumerate}[A.]}
\newcommand{\ENUMAF}{\begin{enumerate}[\bf A.]}
\newcommand{\ENUMi}{\begin{enumerate}[i)]}
\newcommand{\ENDENUM}{\end{enumerate}}
\newcommand{\ITMZ}{\begin{itemize}}
\newcommand{\ENDITMZ}{\end{itemize}}
\newcommand{\REFEQN}[1] { \begin{equation}\label{#1} }
\newcommand{\ENDEQN}{\end{equation}}
\newcommand{\THM}{\begin{theorem}}
\newcommand{\REFEXA}[1] { \begin{example}\label{#1} }
\newcommand{\ENDEXA}{\end{example}}
\newcommand{\REM}{ \begin{remark}}
\newcommand{\ENDREM}{\end{remark}}
\newcommand{\REFTHM}[1] { \begin{theorem}\label{#1} }
\newcommand{\RREFTHM}[2] { \begin{theorem}[#1]\label{#2} }
\newcommand{\ENDTHM}{\end{theorem}}
\newcommand{\REFNTH}[1] { \begin{newthm}\label{#1} }
\newcommand{\ENDNTH}{\end{newthm}}
\newcommand{\REFPROP}[1]{\begin{proposition}\label{#1} }
\newcommand{\RREFPROP}[2]{\begin{proposition}[#1]\label{#2} }
\newcommand{\PROP}{\begin{proposition}}
\newcommand{\ENDPROP}{\end{proposition} }
\newcommand{\REFDEF}[1]{\begin{definition}\label{#1} }
\newcommand{\DEF}{\begin{definition}}
\newcommand{\ENDDEF}{\end{definition} }
\newcommand{\REFLEM}[1]{\begin{lemma}\label{#1} }
\newcommand{\RREFLEM}[2]{\begin{lemma}[#1]\label{#2} }
\newcommand{\LEM}{\begin{lemma}}
\newcommand{\ENDLEM}{\end{lemma} }
\newcommand{\REFCOR}[1]{\begin{corollary}\label{#1} }
\newcommand{\RREFCOR}[2]{\begin{corollary}[#1]\label{#2} }
\newcommand{\RCOR}[1]{\begin{corollary}[#1] }
\newcommand{\COR}{\begin{corollary}}
\newcommand{\ENDCOR}{\end{corollary} }
\newcommand{\COMP}{\begin{complementt}}
\newcommand{\RCOMP}[1]{\begin{complementt}[#1]}
\newcommand{\RREFCOMP}[2]{\begin{complementt}[#1] \label{#2} }
\newcommand{\ENDCOMP}{\end{complementt}}
\newcommand{\REFDEFTHM}[1] { \begin{defthm}\label{#1} }
\newcommand{\ENDDEFTHM}{\end{defthm}}
\newcommand{\corref}[1]{Corollary~\ref{#1}}

\newcommand{\thmref}[1]{Theorem~\ref{#1}}
\newcommand{\propref}[1]{Proposition~\ref{#1}}

\newcommand{\PROOF}{\begin{proof}}
\newcommand{\ENDPROOF}{\end{proof}}

\newtheorem{thm}{Theorem}
\numberwithin{thm}{section}
\newtheorem{compl}[thm]{Complement}
\newtheorem{prop}[thm]{Proposition}
\newtheorem{lem}[thm]{Lemma}

\theoremstyle{definition}
\theoremstyle{definition}
\newtheorem{defn}[thm]{Definition}
\newtheorem{rem}[thm]{Remark}
\newtheorem{cor}[thm]{Corollary}

\begin{document}
\maketitle

\begin{abstract}
This paper concerns the problem of extending the parameter domain of holomorphic motions to include isolated boundary points, punctures of the domains. Supposing that the parameter domain has an isolated boundary point $\lambda^*$, we explore the extension properties of the holomorphic motion to $\lambda^*$. 
\end{abstract}

\section{Introduction}
The notion of holomorphic motion was introduced by Ma\~n\'e-Sad-Sullivan in their work  \cite{masadsul1983} on the dynamics of  rational maps in order to construct conjugacies based on perturbations of the inclusion map of a subset of the sphere. Since then, this concept has been applied as a powerful tool in the field of holomorphic dynamics. 
A holomorphic motion in $\Chat$ can be described as a family of holomorphic functions 
with disjoint graphs over some domain $\Lambda\subset\Chat$ 
parametrized by the fiber over some point $\lambda_0\in\Lambda$. 
More precisely, for a domain $\Lambda$ in $\Chat$, a point $\lambda_0\in\Lambda$, and a subset $E$ of $\Chat$, 
the map $H:\Lambda\times E\rightarrow\Chat$ is called 
\textit{a holomorphic motion of $E$}, parametrized by $\Lambda$, 
with base point $\lambda_0$, if and only if the following conditions hold:
\begin{itemize}
\item[i.] $H(\lambda_0,z)=z$ for all $z\in E$,
\item[ii.] $H$ is vertically injective, i.e., for every fixed $\lambda\in\Lambda$, $H_{\lambda}:z\mapsto H(\lambda,z)$ is injective, and
\item[iii.] $H$ is horizontally holomorphic, i.e., for every fixed $z\in E$, $H^z:\lambda\mapsto H(\lambda,z)$ is holomorphic.
\end{itemize}

There are two famous extension results for holomorphic motions: 
\textit{$\lambda$-Lemma} states that any given holomorphic motion $H:\Lambda\times E\rightarrow\Chat$ has a unique extension to a holomorphic motion of the closure $\overline{E}$, parametrized by $\Lambda$. The second well known result, called \textit{the extended $\lambda$-Lemma} or \textit{S{\l}odkowski's Theorem} states that for simply connected $\Lambda\subset\Chat$, any holomorphic motion $H:\Lambda\times E\rightarrow\widehat{\mathbb{C}}$ extends to a holomorphic motion $H:\Lambda\times \widehat{\mathbb{C}}\rightarrow\widehat{\mathbb{C}}$. 
Moreover, for $\Lambda$ a hyperbolic disk and for every parameter $\lambda\in\Lambda$, 
the extension $H_{\lambda}$ is a quasiconformal homeomorphism, 
whose dilatation is bounded by $e^{d_{\Lambda}(\lambda_0,\lambda)}$, 
where $d_{\Lambda}(\cdot,\cdot)$ denotes the hyperbolic distance in $\Lambda$ \cite{Slodkowski}. 

Let $\Lambda^{*}\subseteq\C$ be a domain with a puncture at $\lambda^{*}$, 
i.e., having an isolated boundary point $\lambda^{*}\in\C$. 
Given a holomorphic motion $H:\Lambda^{*}\times E\rightarrow\C$, when we consider the extension of $H$ to $\lambda^{*}$, either holomorphicity of the functions $H^z$ or injectivity of $H_{\lambda^{*}}$ may fail. Moreover, dynamical information may be lost in such an extension, 
if $H$ is compatible with some dynamics. 
In this paper, we investigate the extension properties of $H$ to 
$\Lambda:=\Lambda^{*}\cup\{\lambda^{*}\}$, when $E$ is connected. We prove that if there exist two distinct points $z_1$ and $z_2$ in $E$ such that both $H^{z_1}$ and $H^{z_2}$ extend holomorphically to $\La$, then for all $z\in E$, $H^{z}$ extends holomoprhically to $\La$. Moreover, we prove that there exist two holomorphic functions $f,g:\La\rightarrow\mathbb{C}$ and a holomorphic motion $\whH$ over $\La$ such that the extension can be written as $H(\la, z)=f(\la)+g(\la)\widehat{H}(\la,z)$. In this respresentation, if $g(\la^*)=0$, then $H_{\la^*}$ is a constant function of $z$, hence the moving set $E$ shrinks to a point at $\la^*$. We call this type of extension \textit{a holomorphic explosion}. In this paper we present our first results in the study of this new notion.

The organization of this article is as follows: In Section \ref{sufcon}, we discuss the extension properties of holomorphic motions to isolated boundary points of the parameter domain. This discussion leads to the introduction of the notion of holomorphic explosions. In the following two sections we give applications of this new concept to holomorphic dynamics.
 
This work was supported by Marie Curie RTN 035651-CODY and Roskilde University.

\section{Main Extension Results of a Holomorphic Motion}\label{sufcon}
Suppose a holomorphic motion $H$ in $\C$ is parametrized by a domain $\Lambda^{*}$, which has an isolated boundary point $\lambda^{*}$ 
such that $\La=\Lambda^{*}\cup\{\lambda^{*}\}\subset\mathbb{C}$. 
We show that under mild conditions, 
the map $H$ \emph{extends horizontally holomorphically} to $\La$. 
That is, each function $H^z$ extends to a holomorphic function 
$H^z : \Lambda \to \C$, and so defines an extension of $H$ to $\La$.
This is the subject of the following theorem.
\begin{thm}\label{thmextension}
Let $E\subset \mathbb{C}$ be connected, and $\Lambda^{*}$ be a domain  with an isolated boundary point $\lambda^{*}$ such that $\La = \La^{*}\cup\{\la^{*}\}\subset\C$. 
Let $H:\Lambda^{*}\times E\rightarrow\mathbb{C}$ be a holomorphic motion 
over $\Lambda^{*}$ with base point $\lambda_0\in\Lambda^{*}$. 
Suppose there exist two distinct points $z_0,z_1\in E$ such that $H^{z_0}$ and $H^{z_1}$ extend holomorphically to $\La$. 
Then $H$ extends horizontally holomorphically to $\La$.
\end{thm}
\begin{proof}
Define
\begin{equation}\label{tildeH}
\wtH(\lambda,z):=z_0+(z_1-z_0)\frac{H(\lambda,z)-H(\lambda,z_0)}{H(\lambda,z_1)-H(\lambda,z_0)}.
\end{equation}
Then $\wtH(\lambda,z_1)\neq \wtH(\lambda,z_0)$ for all $\lambda\in\Lambda^{*}$, 
as $z_1\neq z_0$. 
Thus, for every $z\in E$, $\wtH^{z}=\wtH(\cdot,z)$ is holomorphic in $\Lambda^{*}$. 
Moreover, for every $\lambda\in \Lambda^{*}$, $ \wtH_{\lambda}=\wtH(\la,\cdot)$ is injective, 
since $\wtH_{\lambda}$ is obtained from $H_{\lambda}$ 
by post-composition with an injective affine map, which is the identity for $\la=\la_0$. 
So $\wtH$ is also a holomorphic motion in $\C$ of $E$, parametrized by $\Lambda^{*}$ 
and with base point $\lambda_0$. 
Also $\wtH(\lambda,z_j)\equiv z_j$, $j=0, 1$ by construction. 
Then $\wtH$ fixes the two distinct points $z_0,z_1$ in $\C$. 
Hence, for every $z\in E\sm\{z_0,z_1\}$ the holomorphic function $\wtH^z$ 
takes its values in $\C\sm\{z_0, z_1\}$. So the point $\lambda^{*}$ 
is a removable singularity for $\wtH^z$ by Picard's Theorem. 
That is, $\wtH^{z}$ extends meromorphically to $\La$ for every $z\in E$. 
This defines a horizontally meromorphic extension of $\wtH$ to $\Lambda\times E$.

We show that this extension of $\wtH$ is in fact horizontally holomorphic. 
That is, there is no point $z\in E$, such that $\wtH^z(\la^*) = \infty$. 
Set $U:=\{z\in E;\;\;\wtH(\lambda^{*},z)=\infty\}$. We prove that $U$ is the empty set. 
To this end, we will show that $U$ is both open and closed in the topology of $E$. 
Then as $E$ is connected by hypothesis and $z_0, z_1\in E\sm U$, we conclude that $U=\emptyset$. For each $j$ the subset $E_j := E\sm\{z_j\}$ is a relatively open subset of $E$ containing $U$. Let $\Gamma$ be a round disk in $\La$ with center $\la^*$. 
Define two functions 
$n_j : E_j \to\Z$ for $j=0,1$ by 
\begin{equation*}\label{argprinciple}
n_j(z):=-\frac{1}{2\pi i}\oint_{\partial\Gamma}\frac{\frac{\partial}{\partial\lambda}\wtH(\lambda,z)}{\wtH(\lambda,z)-z_j}d\lambda.
\end{equation*} 
Then each $n_j$ is continuous, since $\partial\Gamma\subset\La^*$, $z_j = \wtH(\la,z_j)$ for all $\la$, 
and $\wtH_\la$ is injective for all $\la\in\La^*$. 
By the Argument Principle, the function $n_j$ counts the number of poles minus the number of zeroes (with multiplicity) of $\lambda\mapsto\wtH(\lambda, z)-z_j$ in $\Gamma$. Hence, each $n_j$ is locally constant in $E_j$ as it takes values in the discrete set $\Z$. 
Furthermore, the only possible zero or pole is at $\lambda^{*}$. 
This means that $n_j(z)$ is the order of $\la^*$ as a pole of $\wtH^z$ if $n_j(z) >0$, 
$-n_j(z)$ is the order of $\la^*$ as a zero of $\wtH^z$ if $n_j(z) < 0$, and 
$n_j(z) = 0$ precisely if $z_j\not=\wtH^z(\la^*)\in\C$. 

To see that $U$ is open, let $z\in U$ be arbitrary. 
Then $z\not= z_0$, $n_0(z) > 0$ and there exists a relatively open 
neighbourhood $\omega\subset E_0$ of $z$ such that 
$n_0(w) = n_0(z) > 0$ for all $w\in\omega$, since $n_0$ is locally constant.
As $z\in U$ was arbitrary, the set $U$ is open in $E\sm\{z_0\}$ and thus in $E$. 

To see that $U$ is closed, we prove that $E\sm U$ is open in $E$. 
Since each $E_j$ is open in $E$ and $E= E_0\cup E_1$ it suffices to prove 
that each subset $E_j\sm U$ is open.
Let $z\in E_j\sm U$ for some $j$. Then $n_j(z)\leq 0$ and there exists a relatively open 
neighbourhood $\omega\subset E_0$ of $z$ such that 
$n_0(w) = n_0(z) \leq 0$ for all $w\in\omega$, since $n_j$ is locally constant.
As $z\in E_j\sm U$ was arbitrary the latter is open in $E_j$ and thus in $E$. 
So $E\sm U = (E_0\sm U)\cup(E_1\sm U)$ is also open. 

To conclude, since $\widetilde{H}$ has a horizontally holomorphic extension to $\lambda^{*}$, so has $H$.
\end{proof}

\begin{compl}\label{complementformofH}
With the hypotheses of Theorem \ref{thmextension}, the map $\wtH$ defined by \eqref{tildeH}  extends to a holomorphic motion of $E$ over $\La$, that is $\wtH_{\la^*}$ is injective. Moreover, there exist two holomorphic functions $f, g:\La\rightarrow\mathbb{C}$ with $f(\la_0)=0$,  and $g(\la_0)=1$, $g$ never vanishing in $\La^{*}$, such that
\begin{equation}\label{formH}
H(\la, z)=f(\la)+g(\la)\cdot\wtH(\la,z).
\end{equation} 
\end{compl}
\begin{proof}
Let us first prove that the horizontally holomorphic extension of $\wtH$ to $\Lambda$, i.e., across $\lambda^*$ is continuous on $\Lambda\times E$. Since it is an extension of a holomorphic motion over $\Lambda^*$ of $E$, we only need to prove continuity at the points $(\lambda^*,z)$, where $z\in E$. However this follows from $\wtH$ being horizontally holomorphic. More precisely, with $\Gamma$ as in the proof above we have 
$$
\wtH(\lambda,z) = \frac{1}{2\pi i}\oint_{\partial\Gamma}\frac{\wtH(\lambda',z)}{\lambda'-\lambda}d\lambda'
$$
Thus, $\wtH$ is continuous $\Gamma\times E$, since it is continuous on $\partial\Gamma\times E$.

Suppose towards a contradiction that there exist two distinct points $w_0, w_1\in E$ such that $\wtH_{\la^*}(w_0) = \wtH_{\la^*}(w_1) = w$. Our aim is to show that $U_w := \{z\in E|\;\;{\widetilde{H}_{\la^*}}(z) = w\}$ is the empty set. To do so, we shall show that $U_w$ is both open and closed in the topology on $E$. Since $\wtH_{\lambda^*}$ is continuous it is closed. Thus we only need to prove it is also open.

Suppose towards a contradiction that $U_w$ is not open. 
Then there exists a point $w_{\infty}\in U_w$ and a sequence $\{w_k\} \subset E\backslash  U_w$ 
such that $w_k\rightarrow w_{\infty}$. 
Moreover, either $w_\infty\not=w_0$ or $w_\infty\not=w_1$. 
We shall detail the first case, the second being similar. Define 
\begin{equation*}
B(\la,z):=\wtH(\la,z)-\wtH(\la,w_0).
\end{equation*}
on $\La\times E$. Clearly, $\wtH(\lambda,z)$ being a holomorphic motion over $\Lambda^{*}$, $B(\lambda,z)$ vanishes at $\lambda^*$ only for $z\in U_w\backslash\{w_0\}$. Let
\begin{equation}\label{balambda}
\epsilon=\inf_{\lambda\in\partial \Gamma}|B(\lambda,w_{\infty})|.
\end{equation}
Since $B(\lambda,z)$ is continuous, and since $w_k\rightarrow w_\infty$, there exists $N\in\mathbb{N}$, such that for all $j\geq N$,
\begin{equation}\label{ba}
\forall\lambda\in\partial \Gamma:\;\;\;\;|B(\lambda,w_{\infty})-B(\lambda,w_j)|<\epsilon.
\end{equation}
By assumption, $B(\lambda,w_j)$ has no zero at $\lambda^{*}$ for all $j\geq N$. 
Considering (\ref{balambda}) and (\ref{ba}),  
$B(\lambda,w_j)$ and  $B(\lambda,w_{\infty})$ have the same number of zeroes in $\Gamma$, 
counted with multiplicity, by Rouch\'e's Theorem. 
This assures that, $B(\lambda, w_j)$ has at least one zero in $\Gamma$, 
which is not $\lambda^{*}$. 
This in turn implies there exists $\lambda'\neq\lambda^{*}$ in $\Gamma$, 
such that $B(\lambda',w_j)=0$, i.e., $\wtH(\lambda',w_j)=\wtH(\lambda',w_0)$. 
But this contradicts the fact that the map $\wtH_{\lambda'}$ is injective for $\lambda'\neq \lambda^{*}$, being a holomorphic motion in $\Lambda^{*}\times E$. 
This proves that there is no sequence $\{w_k\}\subset E\backslash U_w$, 
with $w_k\rightarrow w_{\infty}\in U_w\backslash\{w_0\}$. 
In the case $w_{\infty}=w_0 \in U_w\backslash\{w_1\}$, 
we apply the same procedure with $w_1$ in place of $w_0$ in the definition of $B(\la, z)$
to obtain that there is no sequence $\{w_k\}\subset E\backslash U_w$, 
with $w_k\rightarrow w_{\infty}$. 
Thus $U_w$ is open.

Obviously, the set $E\backslash U_w$ is nonempty, as it contains either $z_0$, $z_1$, or both. Moreover, it is open and closed by the above. Finally, since $E$ is connected, $U_w=\emptyset$.

Solving (\ref{tildeH}) for $H(\la, z)$, we obtain (\ref{formH}) with
\begin{equation}\label{fandg}
\begin{aligned}
f(\la)&:=H(\la,z_0)-\frac{z_0}{z_1-z_0}(H(\la, z_1)-H(\la,z_0))\\
g(\la)&:=\frac{H(\la, z_1)-H(\la,z_0)}{z_1-z_0}.
\end{aligned}
\end{equation}

\end{proof}
\begin{cor}\label{corollaryz_0z_1}
Given $z_0,z_1\in E$:
\begin{itemize}

\item[i.] If $H(\lambda^{*},z_0)\neq H(\lambda^{*},z_1)$, then $g(\la^{*})\neq 0$, and hence the extension of $H$ to $\La\times E$ is a holomorphic motion.
\item[ii.] If $H(\lambda^{*},z_0)=H(\lambda^{*},z_1)=z^{*}\in \mathbb{C}$, then
$g(\la^{*})=0$ and $f(\la^{*})=z^{*}$, that is, for all $z\in E$, $H(\lambda^{*},z)=z^{*}$. 
\end{itemize}

\end{cor}

None of the components on the right side of the form (\ref{formH}) is unique. We will explain this in Proposition \ref{affinerelation} and Corollary \ref{Hsimpleform}, by comparing two representations of $H$ given below.
\\
\\
\textbf{Setup.} Let $\Lambda^{*}$ be a domain  with an isolated boundary point $\lambda^{*}$ such that $\La = \La^{*}\cup\{\la^{*}\}\subset\C$, and let $E\subset \mathbb{C}$ be a set containing at least two points. Suppose that $H:\La^{*}\times E\rightarrow\mathbb{C}$ is a holomorphic motion with base point $\la_0\in\La^{*}$, and with a horizontally analytically extension to $\La\times E$ having two representations:
\begin{eqnarray}
H(\la, z)&=&\widetilde{f}(\la)+\widetilde{g}(\la)\cdot\widetilde{H}(\la,z)\label{wtH}\\
&=&\widehat{f}(\la)+\widehat{g}(\la)\cdot\whH(\la,z)\label{whH}
\end{eqnarray} 
where $\wtH, \whH:\La\times E\rightarrow\C$ are holomorphic motions with base point $\la_0$, $\widetilde{f},\widehat{f}, \widetilde{g}, \widehat{g}:\La\rightarrow \C$ are holomorphic functions with $\widetilde{f}(\la_0)=\widehat{f}(\la_0)=0$, $\widetilde{g}(\la_0)=\widehat{g}(\la_0)=1$, and $\widetilde{g}, \widehat{g}$ never vanishing on $\Lastar$.
\begin{prop}\label{affinerelation}
In the setting above, $\wtH$ and $\whH$ are related by a holomorphically varying family of injective affine maps:
\begin{eqnarray*}
A:\La\times \C&\rightarrow& \C\\
(\la, z)&\mapsto& q(\la)z+r(\la),
\end{eqnarray*}
where $q, r:\La\rightarrow \C$ are holomorphic with $r(\la_0)=0$, $q(\la_0)=1$, $q$ never vanishing on $\La$, such that
\begin{equation*}
    \whH(\la, z)=A(\la, \wtH(\la, z)).
\end{equation*}
\end{prop}
\begin{proof}
Take two distinct points $z_0, z_1\in E$. Then
\begin{eqnarray*}
H(\la, z_1)-H(\la, z_0)&=&\widetilde{g}(\la)\left(\wtH(\la, z_1)-\wtH(\la, z_0)\right)\\
&=&\widehat{g}(\la)\left(\whH(\la, z_1)-\whH(\la, z_0)\right)
\end{eqnarray*}
By the above
\begin{equation*}
\lim_{\la\to\lastar}\frac{\widetilde{g}(\la)}{\widehat{g}(\la)} =
\lim_{\la\to\lastar}\frac{\whH(\la, z_1) - \whH(\la, z_0)}{\wtH(\la, z_1) - \wtH(\la, z_0)} \not = 0.
\end{equation*}
This means that the holomorphic function $q:= \widetilde{g}/\widehat{g}$ extends holomorphically to $\lastar$ with $q(\lastar)\not=0$.

Now consider the equation 
\begin{equation*}
    \widehat{g}(\la)\left(\whH(\la, z)-q(\la)\wtH(\la, z) \right)=\widetilde{f}(\la)-\widehat{f}(\la)
\end{equation*}
which is obtained from equalizing the right sides of (\ref{wtH}) and (\ref{whH}). It follows that the expression $\whH(\la, z)-q(\la)\wtH(\la, z)$ is a holomorphic function of $\la$ on $\La$. Define $r(\la):=\whH(\la, z)-q(\la)\wtH(\la, z)$. Then
\begin{eqnarray}\label{wtHwhH}
    \whH(\la, z)&=&q(\la)\wtH(\la, z)+r(\la),\\
                &=&A(\la, \wtH(\la, z)).     \nonumber
\end{eqnarray}

\end{proof}

\begin{cor}\label{Hsimpleform}

Substituting (\ref{wtHwhH}) in (\ref{whH}) and equalizing the expression with (\ref{wtH}), we obtain
\begin{eqnarray*}
\widehat{g}(\la)&=&q(\la)\widetilde{g}(\la)\\
\widehat{f}(\la)&=&\widetilde{f}(\la)+\widetilde{g}(\la)r(\la).
\end{eqnarray*}

Notice that
\begin{itemize}
    \item The order of $\lastar$ as a zero of $\widetilde{g}$ and $\widehat{g}$ are same, since $q(\lastar)\neq 0$. Let $n$ denote the common order.
     \item The $(n-1)$th Taylor polynomial of $\widetilde{f}$ and $\widehat{f}$ around $\la^{*}$ coincide since the order of $\lastar$ as a zero of  $\widetilde{g}(\la)r(\la)$ is at least $n$. Let  $P(\la-\la^{*})$ denote the common Taylor polynomial.
\end{itemize}

It follows that, given a representation as in (\ref{wtH}), we can always take $\widehat{f}$ and 
$\widehat{g}$  in (\ref{whH}) to be polynomials. Indeed, we have the following very simple form
\begin{equation}\label{formofexplosioneq}
H(\la, z)= P(\lambda-\lambda^{*})+\textstyle{\left(\frac{\lambda-\lambda^{*}}{\laz-\lastar}\right)^n}\cdot\whH(\lambda,z),
\end{equation}
where $\widehat{g}(\la) = \left(\frac{\lambda-\lambda^{*}}{\laz-\lastar}\right)^n$ 
and $\widehat{f}(\la) = P(\la-\lastar)$.
\end{cor}

\rem

The assumption that the set $E$ is connected is essential for the conclusion of Theorem \ref{thmextension} and its complement. 
In fact, let $\La=\D$, $\la^*=0$, so that $\La^* = \Dstar$. Let $\la_0 = 1/e$, set $E_H = \{0,1/e, e\}$, $E_G = \{0, 1/e, 1/e^2\}$ and define two holomorphic motions over $\La^*$:
$$
H(\la,z) = 
\begin{cases}
0,\qquad &z=0,\\
\la,\qquad &z=1/e,\\
1/\la,\qquad &z =e.
\end{cases}
\quad\textrm{and}\quad
G(\la,z) = 
\begin{cases}
0,\qquad &z=0,\\
\la,\qquad &z=1/e,\\
\la^2,\qquad &z =1/e^2.
\end{cases}
$$
Then both $H$ and $G$ satisfies all the hypotheses in Theorem~\ref{thmextension}~except that the sets $E_H$ and $E_G$ are not connected. 
Moreover, the function $H^e$ has a pole at $0$. Thus $H$ extends horizontally meromorphically, but not horizontally holomorphically to $\D$. 
That is, $H$ does not satisfy the conclusion of Theorem~\ref{thmextension}. 
However, $G$ does satisfy the conclusion of Theorem~\ref{thmextension}, but does not satisfy the conclusion of Complement~\ref{complementformofH}.
\endrem

\rem

The assumption in Theorem~\ref{thmextension} that for two distinct points $z_0, z_1\in E$ the holomorphic functions $H^{z_0}$ and $H^{z_1}$ extend holomorphically to $\La$ is essential for the conclusion of the theorem.
Indeed, consider the holomorphic motion $H:\Dstar\times\mathbb{C}\rightarrow \mathbb{C}$ given by
$$
H(\la,z) = e^{(2-1/\la)}\cdot z
$$
with base point $\frac{1}{2}$. 
For this holomorphic motion, $H^{z_0}$, $z_0 = 0$ extends holomorphically to $\lastar = 0$. 
However, $H$ does not satisfy the conclusion of Theorem~\ref{thmextension}.
\endrem

In the following proposition, we give the representation of any holomorphic motion over the whole complex plane, regardless of the connectedness of the moving set.
\begin{prop}\label{1}
Let $E\subset\mathbb{C}$ be a set containing at least two points, and let $H:\mathbb{C}\times E\rightarrow\mathbb{C}$ be a holomorphic motion of $E$ over $\mathbb{C}$ with base point $\lambda_0$. Then there exist two holomorphic functions $f,g:\mathbb{C}\rightarrow\mathbb{C}$ with $f(\lambda_0)=0$, $g$ non-vanishing and $g(\lambda_0)=1$, such that $H$ has an expression
\begin{equation}\label{whenparametersetisC}
H(\la,z)=f(\la)+g(\la)\cdot z
\end{equation}

\end{prop}

\begin{proof}
Let $z_0,z_1\in E$ be distinct points. Define $\widetilde{H}$ as in (\ref{tildeH}).
Recall that $\wtH$ fixes both points $z_0, z_1$. 
Hence for any $z\in E\setminus\{z_0, z_1\}$ the horizontal map 
$\wtH^z:\C\to\C$ avoids the two points $z_0, z_1$, and thus is constant by the Picard's Little Theorem. 
Hence $\widetilde{H}(\la,z) = z$ on $\C\times E$. Taking $f$ and $g$ as in \eqref{fandg} we obtain (\ref{whenparametersetisC}). 
\end{proof}

Obviously, $\wtH$ in (\ref{formH}) can be reparametrized as a holomorphic motion of the set $E^{*}=\wtH(\la^{*}, E)$ over $\Lambda$ with base point $\la^{*}$. Motivated by Theorem~\ref{thmextension}~and Complement~\ref{complementformofH}, we define the notion of \emph{a Holomorphic Explosion}, which is a relative of holomorphic motions.

\begin{defn} (Holomorphic Explosion) Let $\Lambda\subsetneq\Chat$ be a domain and $E^*\subset\mathbb{C}$, let $(\lambda^{*},z^{*})\in\Lambda\times\C$, and let $n>0$ be an integer. 
An order $n$ holomorphic explosion of $E^*$ in $\C$ from $(\lambda^{*},z^{*})$ 
is a horizontally holomorphic map 
$H:\Lambda\times E^*\rightarrow\mathbb{C}$, which can be written as 
\begin{equation}\label{formHexp}
H(\la, z)=f(\la)+g(\la)\cdot\wtH(\la,z),
\end{equation} 
where $\wtH$ is a holomorphic motion of $E^*$ over $\La$ with base point $\la^*$, and $f, g:\La\rightarrow\mathbb{C}$ are holomorphic functions such that $f(\la^*)=z^*$ and $g$ admits $\la^*$ as a zero of order $n$. We call $\lambda^{*}$ \textit{the explosion parameter}. 
\end{defn}

Note that in the definition above 
\begin{itemize}
\item
The motion domain $\La$ can contain more than just one explosion parameter $\la^*$. 
\item Proposition~\ref{affinerelation} applies here. Similar to Corollary~\ref{Hsimpleform}, as an alternative to (\ref{formHexp}), $H$ can be represented as  
\begin{equation}\label{normalizedexplosion}
H(\la, z)= P(\lambda-\lambda^{*})+(\lambda-\lambda^{*})^n\cdot\whH(\lambda,z), 
\end{equation}
where $P(\lambda-\lambda^{*})$ is the $(n-1)$th Taylor polynomial of $f$ around $\lambda^{*}$ and $\whH:\La\times E^{*}\rightarrow\C$ is a holomorphic motion with base point $\lambda^{*}$ which is related to $\wtH$ by a holomorphically varying family of injective affine maps (compare (\ref{formofexplosioneq})).

\item The explosion parameter $\la^{*}$ can be infinity. 
In this case the formula (\ref{normalizedexplosion}) can be replaced by 
$$
H(\la, z)=P\left(\frac{1}{\la-\la'}\right)+\frac{1}{(\la-\la')^{n}}\cdot\whH(\la,z),
$$
where $\la'\notin\La$. 
In particular, if $0\notin\La$ we can take $\la'=0$ and have
\begin{equation*}\label{formHexpatinfty}
H(\la, z)=P\left(\frac{1}{\la}\right)+\frac{1}{\la^{n}}\cdot\whH(\la,z).
\end{equation*} 
\item 

Another more complicated, but also more flexible formulation of 
an order $n$ holomorphic explosion $H$ of the set $E^*$ is as follows. The map $H$ is a horizontally analytic map 
$H:\Lambda\times E^*\rightarrow\mathbb{C}$, 
which can be written as 

\begin{equation}\label{formHexptwo}
H(\la, z)=f(\la)+g(\la)\cdot\whH(\la,\psi(z)),
\end{equation} 
where $\whH$ is a holomorphic motion of a set $E_0\subset\C$ over $\La$ with base point $\la_0\in\La$, $\whH(\la^*,E_0) = E^*$, $\psi(z) = (\whH_{\la^*})^{-1}(z)$ and $f, g:\La\rightarrow\mathbb{C}$ are holomorphic functions such that $f(\la^*)=z^*$, $g$ admits $\la^*$ as a zero of order $n$ and $g(\la_0)\not=0$. 
\end{itemize}

A simple application of the classical $\la$-Lemma shows that a similar conclusion holds for holomorphic explosions, that is, 
the closure $\overline{E^*}$ of $E^*$ in $\mathbb{C}$ also exhibits a holomorphic explosion. 
\begin{lem}\label{lambdalemma}($\lambda$-Lemma for holomorphic explosions) Let $E^*\subset\mathbb{C}$ and let $\La\subsetneq\Chat$ be a domain. Suppose $H:\Lambda\times E^*\rightarrow\mathbb{C}$ is a holomorphic explosion from $(\lambda^{*},z^{*})$. Then $H$ uniquely extends to a holomorphic explosion from $(\lambda^{*},z^{*})$ of the closure $\overline{E^*}$ and parametrized by $\Lambda$.
\end{lem}

\section{Application I}

Consider the one parameter family of transcendental entire maps 
\begin{equation*}
f_a(z)=a(e^z(z-1)+1),\qquad a\in\mathbb{C}^*.
\end{equation*}
This family parametrizes the space of affine conjugacy classes 
of entire transcendental maps with two singular values, one of which coincides with a critical fixed point and the other being asymptotic. 

We denote by $A_a$ the basin of the superattracting fixed point $0$. 
And we denote by $A_a^0$ the immediate basin, that is the connected component of $A_a$ which contains $0$. 
The asymptotic value of $f_a$ is at $z=a$. 
The main hyperbolic component $\mathcal{C}^0$ in parameter space is the set of parameters for which the asymptotic value is contained in $A_a^0$, 
that is,
\begin{equation*}
\mathcal{C}^0=\{a,\;a\in A_a^0\}.
\end{equation*}
\begin{figure}[htb!]
 \begin{center}
\includegraphics[height=5cm,width=5cm]{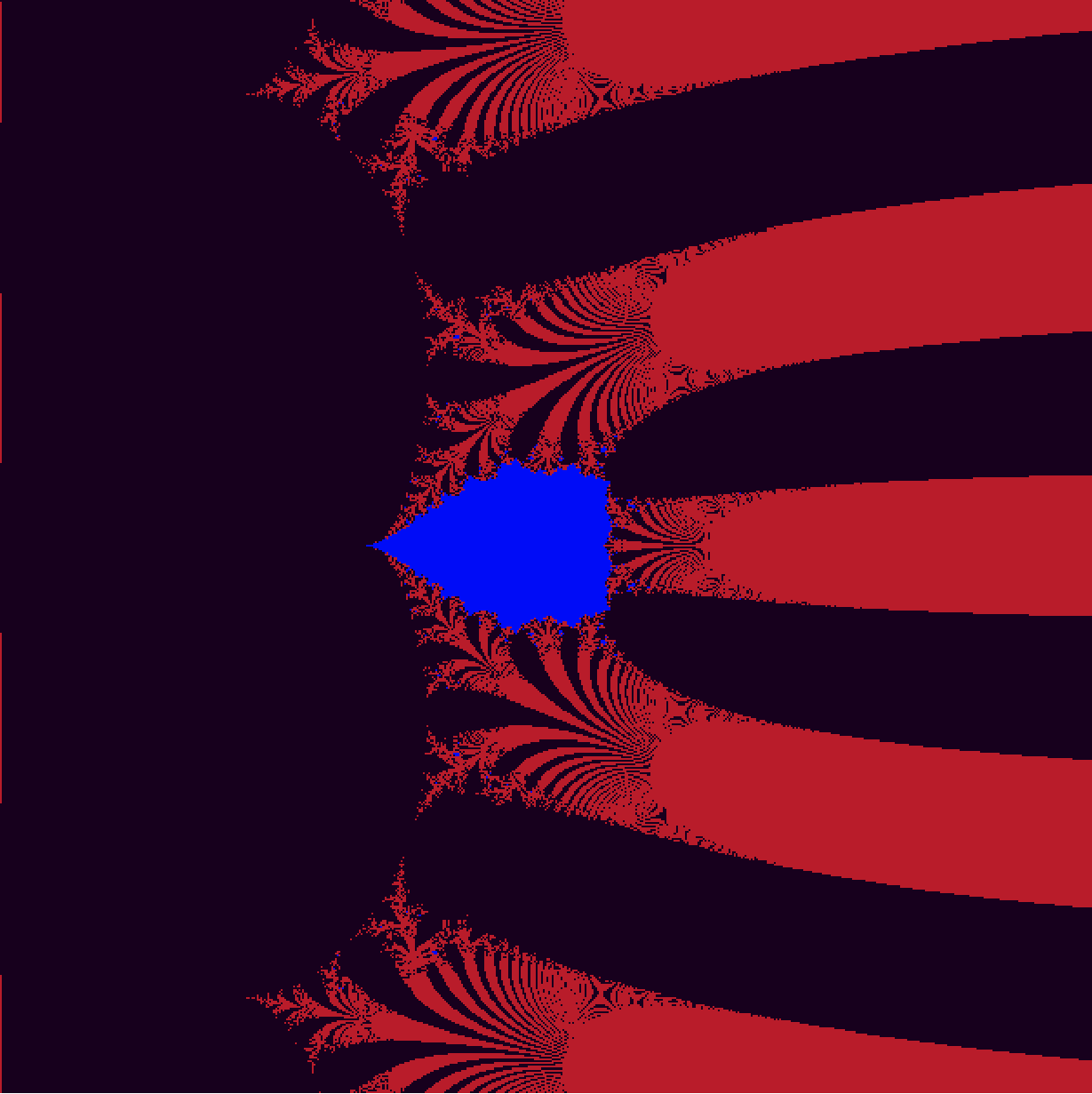}\qquad
\caption{\label{parameterplane_exp}\small{parameter plane in $[10,10]\times[10,10]$ - blue component in the center is $\mathcal{C}^0$.}}\label{parameterplane1}
\end{center}
\end{figure}
It is proven in \cite{deniz2013} that $\mathcal{C}^0$ is bounded, connected and $\mathcal{C}^0\cup\{0\}$ is simply connected.

\begin{figure}[htb!]
\begin{center}
\includegraphics[scale=.25]{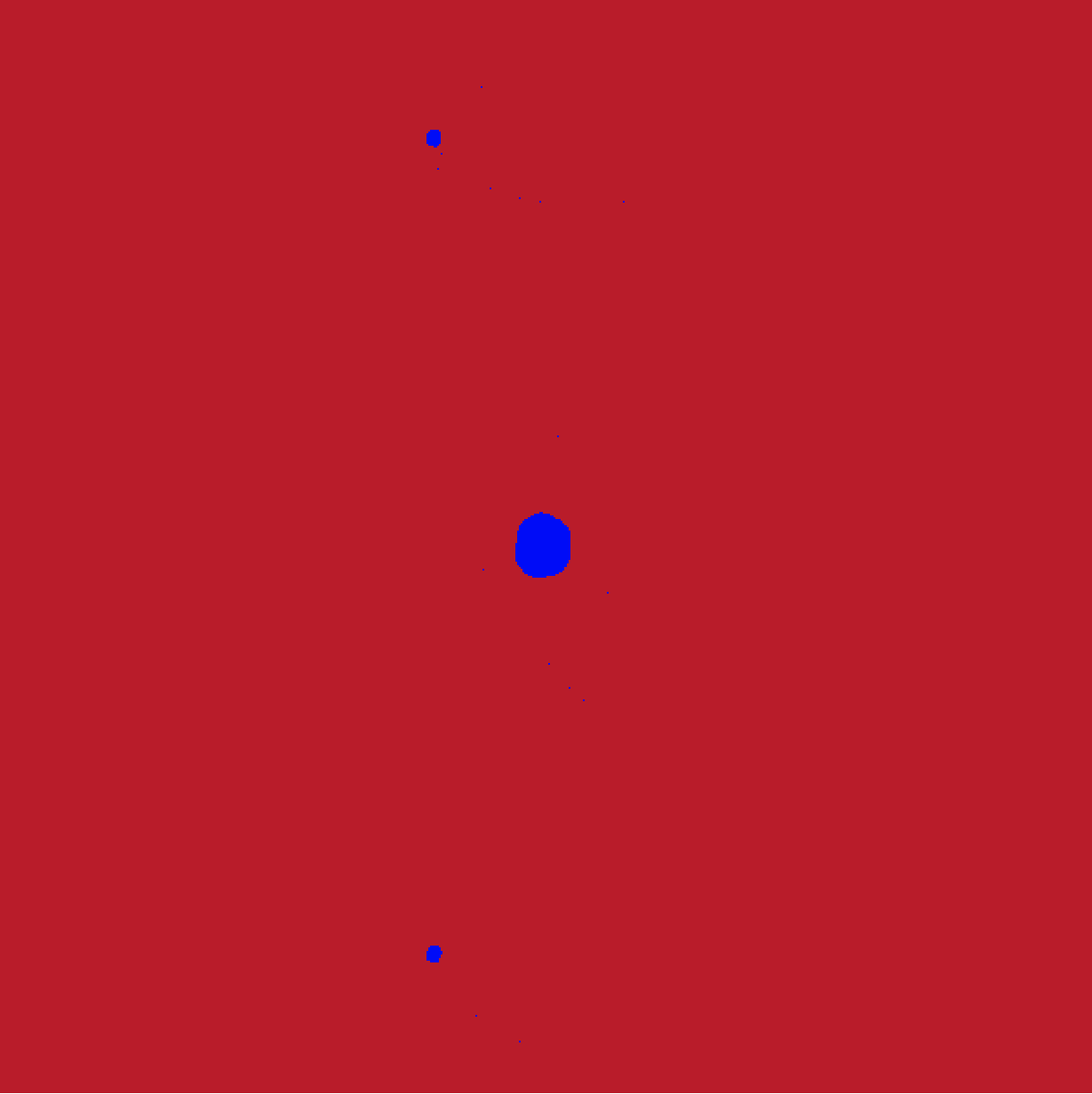}\qquad
\includegraphics[scale=.25]{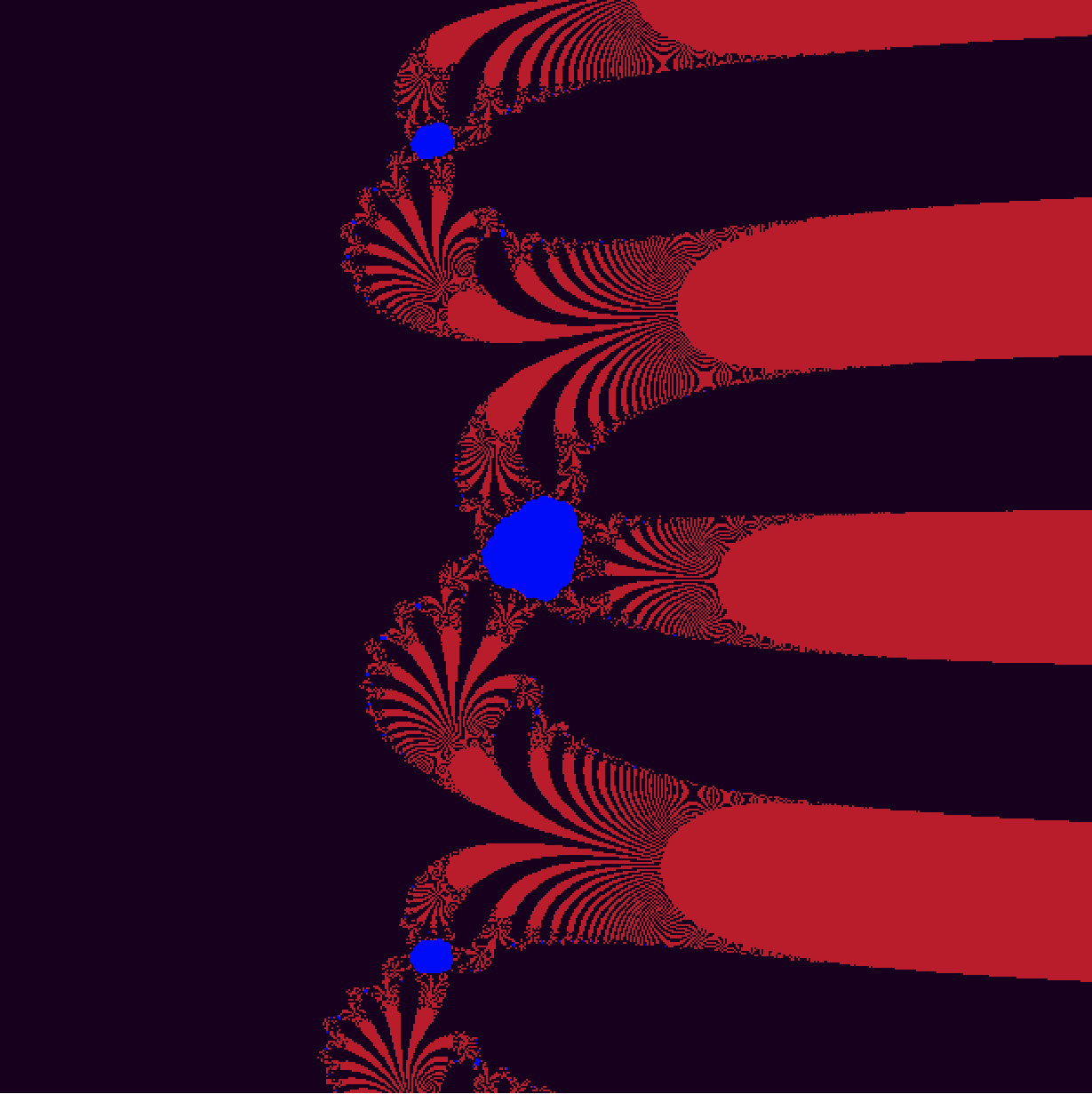}
\caption{\label{dynamicalplanes}\small{left: dynamical plane for $a=3.7+0.5i$; right: dynamical plane for $a=16.33+1.866i$ in $[10,10]\times[10,10]$ - blue component in the center is $A_a^0$. Both dynamical planes are generated with parameters from $\mathbb{C}\backslash\overline{\mathcal{C}^{0}}$.}}
\end{center}
\end{figure}

The B\"{o}ttcher coordinate $\phi_a$ -associated to the fixed point $z=0$- is the unique conformal conjugacy between $f_a$ and $z\mapsto z^2$ 
in some neighborhood of $0$. It has the form
\begin{equation}\label{bot1}
\phi_a(z)=\frac{a}{2}z+O(z^2).
\end{equation}
The conjugacy $\phi_a$ depends holomorphically on $(a,z)$ 
and extends to the whole immediate basin $A_a^0$ 
if and only if $a\notin A_a^0$, or equivalently $a\notin \CC^{0}$. 

Denote by $\psi_a$ the local inverse of $\phi_a$ mapping $0$ to $0$. 
Then for $a\notin\CC^{0}$ the map $\psi_a$ uniquely extends as a biholomorphic map $\psi_a:\D\to A_a^0$ which we denote a B\"{o}ttcher parameter. 

Let $\UU^{*}$ be the unbounded connected component of $\C\sm\overline{\CC^0}$. 
Then the B\"{o}ttcher parameters $\psi_a(z)$ depend holomorphically on $(a,z)\in\UU^*\times\D$.
Fix $a_0\in\UU^{*}$ and define a holomorphic motion $H$ of $A_{a_0}^0$, parametrized by $\UU^{*}$ and with base point $a_0$ 
via the B\"{o}ttcher coordinates and parameters as follows:
\begin{eqnarray}\label{holmotfortheimmediatebasin}
H:\mathcal{U}^{*}\times A_{a_0}^0&\rightarrow&\mathbb{C}\label{eq}\\
H(a,z)&=&\psi_a\circ\phi_{a_0}(z),\;\;\; \notag
\end{eqnarray}
so that $H(a,A_{a_0}^0)=A_a^0$. 
Observe that it follows from \eqref{bot1} that 
\begin{equation}\label{equationofH}
H(a,z)=\frac{a_0}{a}z+O(z^2).
\end{equation}
Moreover, note that $\infty$ is an isolated boundary point of the motion domain $\mathcal{U}^{*}$. Set $\mathcal{U}:=\mathcal{U}^{*}\cup\{\infty\}$.

Our first application of holomorphic explosions is to $H$ in (\ref{holmotfortheimmediatebasin}). Recall that a simply connected domain $A\subsetneq\C$ is called a $K$-quasidisk, if there exists a $K$-quasiconformal map 
$\psi:\C\to\C$ with $\psi(\D) = A$.
\begin{thm}\label{immediatebasinisaquasidisk} For $a\in \mathcal{U}^{*}$, $A_{a}^0$ is a $K$-quasidisk with  
$K:=e^{d_{\mathcal{U}}(\infty,a)}$, where $d_{\mathcal{U}}(\cdot,\cdot)$ denotes the hyperbolic distance in $\mathcal{U}$.

\end{thm}

This theorem strengthens a result in \cite{deniz2013}, where it is shown with a different approach that for all $a_0\in\mathcal{U}^{*}$ the domain $A_a^0$ is a quasidisk, see Figure \ref{dynamicalplanes}.

The proof uses that the holomorphic motion given by (\ref{holmotfortheimmediatebasin}) exhibits a holomorphic explosion at $\infty$. This is the content of the following proposition.

\begin{prop}\label{shrinkto0}
The holomorphic motion $H:\mathcal{U}^{*}\times A_{a_0}^0\rightarrow\C$ can be reparametrised in the vertical direction as a holomorphic explosion

\begin{equation*}
\HH:\mathcal{U}\times \D\rightarrow\C
\end{equation*}
from $(\infty,0)$ of order $1$. 
More precisely, there exists a holomorphic motion 
$\whH:\mathcal{U}\times \D\rightarrow\C$ with base point $\infty$ such that 
$$
H(a,\psi_{a_0}(z)) = \HH(a,z) = \frac{2}{a}\cdot\whH(a,z), 
$$
where $\whH_{a_0} = \psi_{a_0}$ is conformal.
\end{prop}

\begin{proof}
Define $\whH : \UU^*\times\D \to \C$ by $\whH(a,z) := \frac{a}{2}H(a,\psi_{a_0}(z)) = 
\frac{a}{2}\psi_a(z) =: \whpsi_a(z)$. 
We shall show that $\whH$ extends to a holomorphic motion 
of $\D$ over $\UU$ with base point $\infty$. Then as $a\mapsto \frac{2}{a}$ has a simple zero at $\infty$, the proposition will follow.

Notice first that $\whH$ is a holomorphic function of the pair of variables $(a,z)$ 
and that it is vertically univalent with 
$$\whH(a,0) = 0\qquad\textrm{and}\qquad 
\frac{\partial}{\partial z}\whH(a,z)|_{z=0} = \whpsi'_a(0)= 1.
$$
By the Koebe distortion estimates for univalent maps, 
we have 
$$
\forall z\in\D, \forall a\in\UU^*\;\;\; 
|\whpsi_a(z)| \leq \frac{|z|}{(1-|z|)^2}
$$
so that each map $\whH^z$ is uniformly bounded near $\infty$ and so extends holomorphically to $\UU$. That is, $\whH$ extends horizontally holomorphically to $\infty$. Moreover, the univalent functions on $\D$ fixing $0$ with derivative $1$ form a compact family. 
It follows that the extension $\whpsi_\infty := \whH_\infty :\D\to \C$ is univalent. Thus, the extension is also vertically injective. In order to show that $\whH$ is a holomorphic motion with base point $\infty$ we need only to show that $\whpsi_\infty = Id$, or equivalently that the functions $\whpsi_a$ converge uniformly to the identity, as $a\to\infty$. To this end, note that $\whpsi_a$ conjugates $z\mapsto z^2$ to 
$$
g_a(z) := \frac{a}{2}\cdot f_a\left(\frac{2z}{a}\right) = z^2 + \OO(|4z^3/3a|).
$$
In fact, a standard but tedious computation shows that
$$
|g_a(z)-z^2| \leq \left|\frac{4z^3}{3a}\right|e^{|2z/a|}.
$$
So $g_a$ converges to $g_\infty(z) := z^2$ locally uniformly in $\C$. 
By continuity, $\whpsi_\infty$ is the B\"{o}ttcher parameter for $g_\infty$ at $0$ 
and hence $\whpsi_\infty = Id$. 
This completes the proof.
\end{proof}

\begin{proof}[Proof of Theorem \ref{immediatebasinisaquasidisk}]
The holomorphic motion $\whH$ extends to a holomorphic motion of $\Dbar$ 
by the $\lambda$-Lemma and to a holomorphic motion of 
the Riemann sphere $\Chat$ by S{\l}odkowski's Theorem. 
We will denote this extension by $\widehat{H}$ as well.
By the general properties of holomorphic motions 
over a simply connected domain such as $\UU$, 
the real dilatation of the vertical mapping $\whH_a$ is bounded by 
$K_a := e^{d_\UU(a,\infty)}$, when $\infty\in\UU$ is the base point. 
Thus, $A_a^0 = \frac{2}{a}\whH(a,\D)$ is a $K_a$-quasidisk.
\end{proof}
\begin{rem}
The holomorphic explosion $\HH$ extends uniquely to the holomorphic explosion $\HH:\mathcal{U}\times \overline{\D}\rightarrow\C$, by Lemma \ref{lambdalemma}.
\end{rem}

Observe that the solutions of the equation $f_a(z)=0$ do not depend on the parameter $a$. Also, simple calculations show that there are no two solutions with the same imaginary part. These observations allow us to label solutions of $f_a(z)=0$ as $z_j$, $j\in\mathbb{Z}$, ordered in increasing order of $\Im(z_j)$ with $z_0=0$. Let $A_a^{1,j}$ denote the connected component of $f_a^{-1}(A_a^{0})$ which contains $z_j$. We call the point $z_j$, \textit{the center of $A_a^{1,j}$}.

\begin{prop}
For each $j\in\Z\sm\{0\}$, there is a holomorphic explosion $\HH^{j}:\UU\times\Dbar\to\C$ from $(\infty,z_j)$ of order $2$ with 
$$
\HH^{j}(a,\Dbar) = \overline{A}_a^{1,j}.
$$
In particular, all of the components $A_a^{1,j}$ of the basin of $0$ for $f_a$ are also $K$-quasidisks with $K = e^{d_{\UU}(\infty,a)}$.

\end{prop}

\begin{proof}
For each $a\in\UU^*$ and $j\in\Z\sm\{0\}$ let 
$k_{a}^{j}$ be the inverse branch of $f_a$ defined on $\overline{A}_a^0$, which satisfies $k_{a}^{j}(0)=z_j$. 
And define 
\begin{eqnarray*}
\HH^j:\UU\times \D&\rightarrow&\C\\
\UU^*\times\D\ni(a,z)&\mapsto&k_a^{j}\circ \mathcal{H}(a,z)\\
\{\infty\}\times\D\ni(\infty,z)&\mapsto&z_j.
\end{eqnarray*}
This gives the desired holomorphic explosions. From the series form of $\HH^{j}$ near $0$
\begin{equation*}
\mathcal{H}^{j}(a,z)=k_a^{j}(\mathcal{H}(a,z))=z_j+\frac{2}{a^{2}e^{z_j}z_j}\widehat{H}(a,z)+O\left(\frac{1}{a^{4}}\right),
\end{equation*}
we see that the order is $2$. Here $\whH$ is as in \propref{shrinkto0}.
\end{proof}

\section{Application II }

We are interested in the moduli space $\MMtwo$ of quadratic rational maps modulo 
{\Mobius} conjugacy. Every quadratic rational map $R$ has three 
fixed points counting with multiplicity, with multipliers which we generically denote by $\la, \mu$ and $\ga$. 
Denote by $\si_1(R) = \la +\mu + \ga$, $\si_2(R) = \la\mu+\la\ga+\mu\ga$ 
and $\si_3 = \la\mu\ga$ the three elementary symmetric functions of $\la$, $\mu$ and $\ga$.
Following Milnor \cite{milnor1993}, we endow $\MMtwo$ with the affine structure given by $(\si_1,\si_2)$. In more detail 
$\si_3(R) = \si_1(R)-2$, and any 
pair of numbers $(\si_1,\si_2)\in\C^2$ defines a unique equivalence class in $\MMtwo$, so that $(\si_1,\si_2):\MMtwo\to\C^2$ is biholomorphic  \cite[Lemma 3.1]{milnor1993}. For each $\la\in\C$ the curve 
$$
\Per_1(\la) = \{\; [R]\; \bigm|\; R~\textrm{has a fixed point of multiplier}~\la\;\}.
$$
is a straight line in this affine structure, \cite[Lemma 3.4]{milnor1993}. Moreover, $\si(\la,R):= \mu(R)\ga(R)$ defines an isomorphism between $\Per_1(\la)$ and $\C$ \cite[Remark 6.9]{milnor1993}. 

For $\la\in\D$ we consider the connectedness locus $\Mla$ in $\Per_1(\la)$
$$
\Mla = \{\; [R] \in\Per_1(\la)\;\bigm|\; J_R \textrm{ is connected }\},
$$
where $J_R$ denotes the Julia set of $R$. For each $R$ with $[R]\in \Mla$, there exists an isomorphism $\psi$ from the complement of the filled in Julia set to $\widehat{\mathbb{C}}\backslash\overline{\mathbb{D}}$ satisfying $\psi(\infty)=\infty$. Then the conjugate map $\psi\circ R\circ\psi^{-1}=:B_{\la}:\widehat{\mathbb{C}}\backslash\overline{\mathbb{D}}\rightarrow \widehat{\mathbb{C}}\backslash\overline{\mathbb{D}}$ is a quadratic Blaschke product for which $\infty$ is a fixed point with multiplier $\la$. Possibly post-composing $\psi$ with a rigid rotation, we obtain $B_\la(z) = z\frac{z+\labar}{1+\la z}$. The restriction of $B_{\la}$ to the  the unit circle is called a representative of the external class of $R$.

The special line $\Per_1(0)$ also equals the moduli space for quadratic polynomials 
$$
\Per_1(0) = \{\; [Q_c]\;\bigm|\; Q_c(z) = z^2 + c, c\in\C\;\}
$$
for which $\si(Q_c)= 4c$, so that $\si(\Mbrot_0) = 4\cdot\Mbrot$, where $\Mbrot$ is the Mandelbrot set. 

\begin{thm}\label{MQRapp}
There is a natural dynamically defined holomorphic motion 
{\mapfromto H {\D\times(4\cdot\Mbrot)} \C} over $\D$ with base point $0$ 
of the Mandelbrot set scaled by a factor $4$ such that:
\ENUM
\item
For each $\la\in\D$ : $H(\la,\si(0,\Mbrot_0)) = \si(\la, \Mla)$.
\item
For each $c\in\Mbrot$ and $\la\in\mathbb{D}$, any quadratic rational map $R\in\si_{\la}^{-1}(H(\la,4c))$ 
has a polynomial like restriction, hybridly equivalent to a restriction of $Q_c$. 
\ENDENUM
\end{thm}

For $\la\in\Cstar$ and $B\in\C$ the quadratic rational map 
$z\mapsto \frac{1}{\la}(z+B+1/z)$ admits $\infty$ 
as a fixed point with multiplier $\la$, and thus represents an element of $\Per_1(\la)$. 
Moreover, the maps $z\mapsto \frac{1}{\la}(z+B+1/z)$ and $z\mapsto \frac{1}{\la}(z-B+1/z)$ are conjugate by $z \mapsto -z$, and thus represent the same element of $\Per_1(\la)$. 
Therefore, for each $\la\in\Cstar$ and $A\in\C$ such that $B=\sqrt{A}$, the map
$$
(\la,A)\mapsto [\GlaA], \qquad \GlaA(z) := \frac{1}{\la}(z+\sqrt{A}+1/z)
$$
is well defined without specifying which square root is being used. When a specific choice of square root is relevant, it will be clear from the context how it is chosen. The maps $\GlaA$ admits $\pm 1$ as critical points, and a simple computation 
shows that the product of the two remaining fixed point eigenvalues is given by 
\REFEQN{sigmaArelation}
\sibf_{\la}(A) = \sibf(\la,A) = \si(\lambda,[\GlaA])=\si_{\la}([\GlaA]) = \frac{(\la-2)^2-A}{\la^2}
\ENDEQN
\cite[page 73, column 1]{milnor1993}(Milnor uses the symbol $\tau$ for $\si$, $a^2$ for $A/\la^2$ and $\mu$ for $\la$).
It follows that for each fixed $\la\in\Cstar$ the map
$$
A\mapsto [\GlaA] : \C \longrightarrow \Per_1(\la)
$$
is an isomorphism. We denote by $\Ala$ the inverse of this isomorphism and for $\la\in\Dstar$ we write 
\begin{equation*}
\whMbrot_\la := \Ala(\Mbrot_\la)
\end{equation*}
for the connectedness locus of the family $\GlaA, A\in\C$. Equivalently $\whMbrot_\la$ is the set of parameters $A$ for which, either $1$ or $-1$ has bounded orbit under $\GlaA$. Note that for any $\la$ the map $G_{\la,0}$ is conjugate to itself under $z\mapsto -z$ and thus, $0\notin\whMbrot_\la$ for any $\la\in\Dstar$. In Figure \ref{connectednessloci}, connectedness loci are shown for different $\lambda$ values. 

\begin{figure}[htb!]
\begin{center}
\includegraphics[height=7cm,width=7cm]{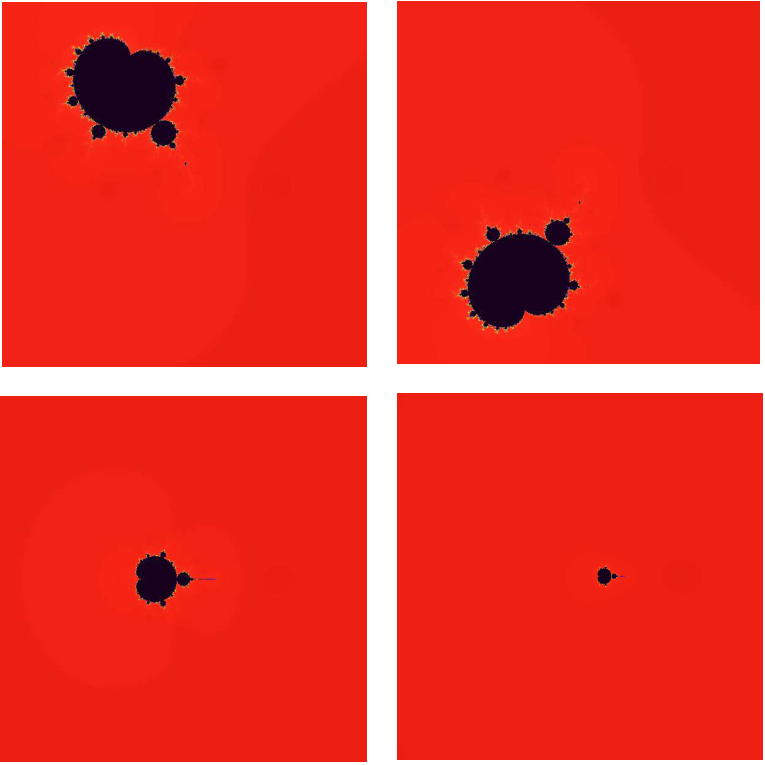}
\caption{\label{connectednessloci}{Connectedness loci in $A$-plane with $\lambda$ parameter values top left: $0.4-0.35i$; top right $0.4+0.35i$; bottom left $e^{-1}$; bottom right $0.2+0.2i$, in $[1,5]\times[-2,2]$.}}
\end{center}
\end{figure}

Since the external class of the maps $R$ with $[R]\in\Per_1(\la)$ is $B_\la$, it follows from Yoccoz inequality for polynomial-like maps, \cite[Cor. D2]{Petersen} that for any $r<1$ the map $\si_{\la}$ is bounded on 
$$
\bigcup_{\la, ||\la|| < r}\Mbrot_\la.
$$
It thus follows immediately from (\ref{sigmaArelation}) that $\textbf{A}_\la$ converge to $4$ uniformly on $\Mbrot_\la$, when $\la\to 0$. 

Holomorphic motion of the connectedness loci $\whMbrot_{\la}$ with the parameter $\lambda$ was the subject of the master's thesis of Uhre \cite{uhre2004}. We adopt the following result from her thesis. The interested reader can see \cite[Chapter 8.1-8.3]{uhre2004} for a proof. 

\begin{thm}\label{thmevamot}
There exists a dynamically defined holomorphic motion of the connectedness locus $\whMbrot_{e^{-1}}$ 
over $\Dstar$ with base point $e^{-1}$:
$$
\mathbb{A}: {\Dstar\times\whMbrot_{e^{-1}}}\rightarrow\C\qquad\textrm{with}\qquad
\mathbb{A}_\la(\whMbrot_{e^{-1}})=\whMbrot_{\lambda}.
$$

More precisely, for each fixed pair $(\la, A)$ with $\lambda\in\Dstar$ and $A\in\whMbrot_{e^{-1}}$, 
there exists a global quasiconformal mapping 
$h_{\lambda,A}:\Chat\rightarrow\Chat$ 
conjugating $G_{e^{-1},A}$ to $G_{\lambda,\AAA(\lambda,A)}$ and which restricts to  a hybrid conjugacy between quadratic like restrictions of the two maps.

\end{thm}

In \cite{uhre2004}, the proof was given for the connectedness locus for the family given by the formula 
\begin{equation*}
R_{\la,\mu}(z)=z\frac{z+\mu}{1+\la z}
\end{equation*}
which is conjugate by the affine map 
\begin{equation*}
\phi_{\la,\mu}(z)=\frac{\la z+1}{\sqrt{1-\la\mu}}.
\end{equation*}
to $G_{\la,A}$, 
where 
\begin{equation}\label{my-A-relation}
A=(\la-2)^{2}-\la^{2}\mu\frac{2-\la-\mu}{1-\la\mu}
\end{equation}
In the formula above, observe 
\begin{equation*}
\mu\frac{2-\la-\mu}{1-\la\mu} = \si_{\la}([G_{\la, A}])
\end{equation*}
i.e., the product of the two other (than $\la$) fixed point eigenvalues, compare (\ref{sigmaArelation}) and (\ref{my-A-relation}).

\begin{proof}[Proof of Theorem \ref{MQRapp}]
This is now a straight forward application of \corref{corollaryz_0z_1} to the result of \thmref{thmevamot}.
This approach provides an alternative  to \cite[Chapter 8.4]{uhre2004}. 
Notice that the extension to $0$ corresponds to  the case where $\infty$ is a superattracting fixed point. 

Since $\sibf_\la(A) = \sibf(\la,A) = \si(\la,[\GlaA]) = 
\si(\la,\Ala^{-1}(A)) = \si_\la(\Ala^{-1}(A))$ is holomorphic as a function of the two variables 
and injective, in fact, affine in the second variable (see \eqref{sigmaArelation}), the holomorphic motion $\AAA$ 
is equivalent to the holomorphic motion of $\si_{e^{-1}}(\Mbrot_{e^{-1}})$ over 
$\Dstar$ with base point $e^{-1}$:
\REFEQN{sigmaDstarmotion}
{\mapfromto \wtH {\Dstar\times\si_{e^{-1}}(\Mbrot_{e^{-1}})} \C}
\qquad 
\wtH(\la,\tau) = \sibf_\la\circ\AAA(\la,\sibf_{e^{-1}}^{-1}(\tau))
\ENDEQN
This relation is illustrated by the following diagram:
\begin{center}
\begin{tikzcd}
\whMbrot_{e^{-1}}\arrow{r}{\mathbb{A}}\arrow[dd,bend right=90, swap,  "\sibf_{e^{-1}}"] &\whMbrot_{\la}\arrow[dd,bend left=90, pos=0.49, "\sibf_{\la}"] \\   
\Mbrot_{e^{-1}}\arrow{u}{\bm{A}_{e^{-1}}}\arrow{d}{\si_{e^{-1}}} & \Mbrot_{\lambda}\arrow{u}{\bm{A}_{\lambda}}\arrow{d}{\si_{\lambda}}\\
\si_{e^{-1}}(\Mbrot_{e^{-1}})\arrow{r}{\wtH}   &\si_{\lambda}(\Mbrot_{\lambda})
\end{tikzcd}
\end{center}

The holomorphic motion $\wtH$ has two special points, whose motions are easy to follow. 
The center, for which one of the fixed points is persistently superattracting, i.e., it has multiplier $0$ so that this trajectory is the constant $0$. 
The other is the root corresponding to the two remaining fixed points persistently coalesce to form a parabolic fixed point of multiplier $1$ so that this trajectory is the constant $1$. That is, the motions of the points $0$ and $1$ under $\wtH$ are constants. Thus, by \corref{corollaryz_0z_1}, $\wtH$ extends to a holomorphic motion over $\D$. We denote this extension also by $\wtH$. We write $E=\wtH(0,\si_{e^{-1}}(\Mbrot_{e^{-1}}))$ and denote by 
{\mapfromto H {\D\times E} \C} its reparametrization from $0$:
\REFEQN{thesigmamotion}
H(\la,\tau) = \wtH(\la,\wtH_{0}^{-1}(\tau)).
\ENDEQN
To prove \thmref{MQRapp} we just need to show that in fact, $E = 4\cdot\Mbrot$. 
To see this, note that $[\GlaA]$ contains both of the maps $R_{\la,\mu}$ for which $\mu$ is a solution of the equation \eqref{my-A-relation}, which is quadratic in $\mu$. When $\la\to 0$ the maps $R_{\la,\mu}$ converge to quadratic polynomials of the form $\mu z+ z^2$ uniformly on the sphere. 
For any fixed $\tau=\sibf_{e^{-1}}(A')\in\si(e^{-1},M_{e^{-1}})$ and any $\la\in\Dstar$, let $\mu(\la,\tau)$ be a solution of 
the quadratic equation \eqref{my-A-relation} with $A = \mathbb{A}(\la, A')$. Then the class $[G_{\lambda,\AAA(\lambda,A')}]$ 
contains $R_{\la,\mu(\la,\tau)}$. Thus, for a continuous choice of $\mu(\la,\tau)$ this function converges to $\mu_0$ such that for the quadratic polynomial $\mu_0z+z^2$ the product of the two finite multipliers is $\tau_0=\wtH(0,\tau)$. 
This polynomial is affinely conjugate to $z^2 + c$ with 
$4c = \tau_0$ and has connected filled-in Julia set. 
Thus, $E\subset 4\cdot\Mbrot$. Furthermore, this polynomial is the unique quadratic polynomial for which $R_{\la,\mu(\la,\tau)}$ and $G_{\lambda,\AAA(\lambda,A')}$ are hybridly equivalent according to \cite[Prop. 14 page 313]{douadyandhubbard}.

To complete the proof of \thmref{MQRapp} we need to show surjectivity, i.e, $E=4\cdot\Mbrot$. 
For this we appeal to the \propref{general_straightening} below.
It follows from the proposition that for every $\tau_0=4c \in4\cdot\Mbrot$ there exists $A'\in\whMbrot_{e^{-1}}$ such that $G_{e^{-1},A'}$ is hybridly equivalent to $Q_c$. 
Thus, $\tau_0\in E$, i.e., $4\cdot\Mbrot\subset E$.
\end{proof}
\begin{prop}\label{general_straightening}
For every $c\in\Mbrot$ there exists $A\in\whMbrot_{e^{-1}}$ such that 
$Q_c$ and $G_{e^{-1},A}$ are hybridly equivalent.
\end{prop}

\begin{proof}
This is a particular case of \cite[Prop. 5 page 301] {douadyandhubbard} applied to a quadratic-like restriction of $Q_c$ and the external class $B_{e^{-1}}$.
\end{proof}

Then reapplying \eqref{sigmaArelation} again or applying \corref{corollaryz_0z_1} 
directly to $\AAA$ we find 
\begin{cor}
The holomorphic motion $\AAA$ of $\whMbrot_{e^{-1}}$ over $\Dstar$ 
is a quasiconformal reparametrization of a holomorphic explosion over $\D$ from $(0,4)$. 
Indeed, writing $A(\tau) = \sibf_{e^{-1}}^{-1}\circ H({e^{-1}},\tau)$, 
then for any $\tau\in \Mbrot_0$ we have 
\REFEQN{MbrotAext}
\AAA(\la,A(\tau)) = (\la-2)^2 - \la^2 H(\la,\tau),
\ENDEQN
where $H$ is the holomorphic motion in \thmref{MQRapp} as given by \eqref{thesigmamotion}.

\end{cor}

\end{document}